\documentclass[preprint,12pt]{elsarticle2}
\usepackage{graphicx}
\usepackage{amssymb}
\usepackage{amsmath}
\usepackage{url}

\usepackage{tikz}
\usetikzlibrary{arrows}

\def\be{\begin{equation}}
\def\ee{\end{equation}}

\newtheorem{theorem}{Theorem}[section]
\newtheorem{lemma}[theorem]{Lemma}

\newtheorem{definition}[theorem]{Definition}

\newenvironment{proof}{\noindent{\bf Proof \/}}{\hfill $\Box$\vskip 0.1in}

\catcode`\@=11

\def\Expect{\mathbb{E}~}
\def\Prob{\mathbb{P}}

\newcommand{\bfalpha}{\mbox{\boldmath$\alpha$}}

\newcommand{\Exp}{{\mathbb E}\,}

\newcommand{\RL}{{\mathbb R}}

\newcommand{\downleftarrow}{\smash{%
\raisebox{.25ex}{%
\setlength{\tabcolsep}{-1.2pt}%
\begin{tabular}{@{}lr@{}}%
\multicolumn2r{$\downarrow$} \\[-3.8ex] \multicolumn2c{$\leftarrow\,$}%
\end{tabular}}}}

\newcommand{\uprightarrow}{\smash{%
\raisebox{-.25ex}{%
\setlength{\tabcolsep}{-1.2pt}%
\begin{tabular}{@{}lc@{}}%
\multicolumn1l{$\uparrow$} \\[-1.8ex] \multicolumn2c{$\,\rightarrow$}%
\end{tabular}}}}

\begin{document}
\begin{frontmatter}

\title{Non-Existence of Stabilizing Policies\\
for the Critical Push-Pull Network and Generalizations}

\author[UQ]{Yoni Nazarathy}
\author[UQ]{Leonardo Rojas-Nandayapa}
\author[YK]{Thomas S. Salisbury
\footnote{Research supported in part by NSERC.}
}

\address[UQ]{School of Mathematics and Physics, The University of Queensland, Brisbane, Australia.}
\address[YK]{Department of Mathematics and Statistics, York University, Toronto, Canada.}


\date{ \today}

\begin{abstract}
The push-pull queueing network is a simple two server, two job-stream example in which servers either serve jobs from queues or generate new arrivals. Previous work has shown that there exist non-idling policies that stabilize the system in the positive recurrent sense for all parameter settings in which the network may be rate stable, except for the case where processing rates are equal on each job stream (critical).  It was conjectured in Kopzon, Nazarathy and Weiss (2009) that there is no policy that makes the network positive recurrent (stable) in the critical case. Our contribution here is a proof for that conjecture. We also consider generalizations where it is shown that a stabilizing non-idling policy does not exist in the critical case. In this respect we put forward a general sufficient condition for non-stabilizability of queueing networks.
\end{abstract}

\end{frontmatter}

\section{Introduction}

Controlled queueing networks are primary objects of study in operations research and applied probability as they provide sensible models for a variety of engineering, communications and service situations. Alongside performance analysis and optimal control, stability analysis plays a central role in the theory and has far reaching implications in the design of systems.   A comprehensive introduction to stability properties of controlled queueing networks is \cite{bramsonBook2008}. See also \cite{bookChenYao2001} and~\cite{bookMeyn2008}.

The push-pull queueing network, introduced in \cite{WeissKopzon112}, is perhaps the simplest example of a queueing network that generates its own input and is able in certain cases (described below) to operate with servers fully utilized while keeping queues stochastically stable. Following \cite{WeissKopzon112}, the network was analyzed with exponential processing times in \cite{kopzon2009push} and general processing times in \cite{nazarathy2010positive}. Further network generalizations are in \cite{GuoLefeberNazarathyWeissZhang11}. The novelty of this network is that by allowing servers to split processing power between ``arrival generation'' and ``service of current jobs'', one can often find stabilizing control policies in which servers never idle. This stands in contrast to the majority of queueing network models (surveyed in \cite{bramsonBook2008}, \cite{bookChenYao2001} and~\cite{bookMeyn2008}) in which high utilization cannot be achieved without having to endure high congestion levels.

\begin{figure}[ht]
\centering
\includegraphics[width=3.5in]{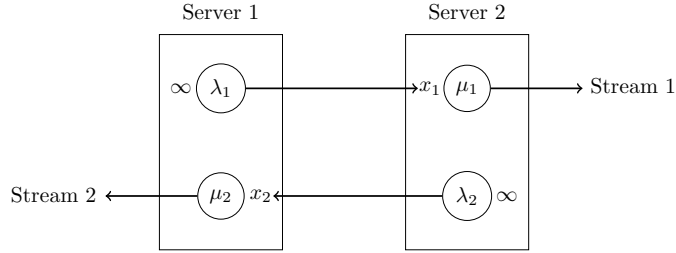}
\caption{The push-pull network: Two servers are working non-stop on two job streams, $i=1,2$. Push operations, labeled $\lambda_i$, move jobs to the queues $x_i$ served by pull operations labeled $\mu_i$.} 
\label{fig.PushPull}
\end{figure}

The push-pull network is illustrated in Figure~\ref{fig.PushPull}. Two servers are working on two job streams. Each job stream begins with a {\em push} operation (labeled $\lambda_i,\, i=1,2$) and follows with a {\em pull} operation (labeled $\mu_i,\, i=1,2$).  The push operation produces jobs, serving a virtual unlimited supply of raw materials, and moves them down to the pull operation. The pull operation is associated with a queue (labeled $x_i,\, i=1,2$). At any state, a scheduling policy (control) determines for each server whether it should work on {\em push} or {\em pull}. Pushing may always be performed as it is assumed that there is an infinite supply of raw materials at the start of each stream. As opposed to that, pulling may only be performed if the associated queue is non-empty.  We are interested in non-idling (fully utilizing) policies.

It is standard to 
associate the network with a probability space. Specifically, four independent $i.i.d.$ sequences of non-negative random variables with positive finite mean are taken as primitives of the construction. These random variables signify the durations of consecutive operations. The strictly positive parameters $\lambda_i$ and $\mu_i$ are the processing rates of these operations (i.e.\ inverses of mean durations). A full construction of the model is described in \cite{nazarathy2010positive}. 

In this paper we shall impose a simplifying assumption in which the random variables have a memory-less exponential distribution. In this case, if preemption of operations is allowed and it is assumed that job durations are not known upon commencement, then the state space of the system can be represented by ${\cal S} = \mathbb{Z}_+^2$, where a state $(x_1,x_2) \in {\cal S}$ implies that there are $x_i$ jobs in the queue $i$. In our setting, any deterministic and {\em non-idling} control policy can be represented by a function,
\[
{\cal P} : {\cal S} \rightarrow \{\mbox{push},\mbox{pull}\}^2,
\]
whose image represents the actions of the servers: $\big(\mbox{action server 1},\mbox{action server 2} \big)$. Since idling is not an option, the control should satisfy for any integer $x\ge0$ the following conditions
\[
{\cal P}\big(  (x,0) \big) \in   \{\mbox{push}\}  \times\{\mbox{push},\mbox{pull}\}, 
\quad
{\cal P}\big(  (0,x) \big) \in   \{\mbox{push},\mbox{pull}\} \times \{\mbox{push}\}.
\]
Thus it is also implied that ${\cal P} \big( (0,0) \big)  = (\mbox{push},\mbox{push})$.\

Given a policy ${\cal P}$, the evolution of the network is well represented by a Markov jump process
(see for example \cite{bookAsmussen2003}, Chapter~II). 
In this case it is natural to define the following stability properties.
\begin{definition}
\label{def:nonStab}
A network controlled by a policy ${\cal P}$ is said to be {\bf stable} if the associated Markov jump process has a positive recurrent class such that the process enters this class with probability~1. Further, a given network is said to be {\bf stabilizable} if there exists a policy ${\cal P}$ for which the network is stable. Otherwise the network is {\bf non-stabilizable}.
\end{definition}
We emphasize that the above definition relates to policies in which the servers are fully utilized. Achieving stability when idling is allowed (no pushing) is trivial.

In \cite{kopzon2009push} (see also \cite{nazarathy2010positive} for general processing times), the authors have shown that the push-pull network is stabilizable in the case $\lambda_i < \mu_i$, $i=1,2$ or in the case $\lambda_i > \mu_i$, $i=1,2$. Further, it is obvious by capacity considerations that there is no such policy if $\lambda_1 < \mu_1$, $\lambda_2 > \mu_2$, or in the alternative where the indexes are switched. The remaining open question is the critical case, 
\[
\lambda_i = \mu_i,\,\, i=1,2.
\]
In that case, while there exist simple rate-stable policies in which the associated Markov jump process is null recurrent (see Theorem~1 in \cite{kopzon2009push} and the discussion in Sections 1 and 2 in \cite{nazarathy2010positive}), it was conjectured in \cite{kopzon2009push} (page 83) that the network is non-stabilizable. Our key contribution in this note is in settling that conjecture. Our proof is based on a simple yet non-trivial martingale stopping argument in which a linear martingale is constructed for {\em any} possible policy. 

A second contribution of the paper is that we are able to extend the method for proving non-stabilizability  to more general networks. The novelty of our method is that it puts forward a simple matrix rank criterion as a sufficient condition required for a network to be non-stabilizable. As we show, this condition turns out to be powerful enough to be applied in various types of queueing networks that generate their own input. We prove non-stabilizability for two structured generalizations of the push-pull network in certain critical cases.  First we handle a ring with an even number of push-pull servers (Figure~\ref{fig.ring}). Second we handle a two server network with multiple re-entrant lines (Figure~\ref{fig.rlinePushPull}). Stable policies have been found for both of these types of models for certain non-critical parameter settings in \cite{GuoLefeberNazarathyWeissZhang11}. 

The remainder of this paper is structured as follows. In Section~2 we prove that the critical push-pull is non-stabilizable. In Section~3 we prove similar results for two generalizations: A ring with an even number of critical push-pull servers and a two server network with re-entrant lines. Closing comments are in Section~4.

\section{The Critical Push-Pull is Non-Stabilizable}
Note that a policy ${\cal P}$ induces a random walk on $\mathbb{Z}_+^2$ with transitions of the form $\big(\, \uprightarrow\, , \downleftarrow\, , \leftrightarrow, \updownarrow\big)$ corresponding to
\begin{equation*}
\big((\mbox{push},\mbox{push}), (\mbox{pull},\mbox{pull}), (\mbox{push},\mbox{pull}),(\mbox{pull},\mbox{push})\big).
\end{equation*}  
This is a discrete time embedded Markov chain, $\{X_n,n\ge0\}$ with transition probabilities, $p_{(x_1,x_2), (x'_1,x'_2)} = \Prob\big(X_{n+1} = (x'_1,x'_2) \,|\, X_{n} = (x_1,x_2)\big)$ defined as follows:
\begin{align*}
\label{eq:dynamics}
p_{(x_1,x_2), (x_1+1,x_2)}& = \dfrac{\lambda_1}{\lambda_1+\lambda_2}, &
p_{(x_1,x_2), (x_1,x_2+1)}& = \dfrac{\lambda_2}{\lambda_1+\lambda_2}, &
\,\,\,\mbox{if}  \quad {\cal P}\big((x_1,x_2)\big)& = (\makebox[0.8cm][c]{push},\makebox[0.8cm][c]{push}),
\\[5pt]
p_{(x_1,x_2), (x_1-1,x_2)}& = \dfrac{\mu_1}{\mu_1+\mu_2}, &
p_{(x_1,x_2), (x_1,x_2-1)}& = \dfrac{\mu_2}{\mu_1+\mu_2}, &
\,\,\,\mbox{if}  \quad {\cal P}\big((x_1,x_2)\big)& = (\makebox[0.8cm][c]{pull},\makebox[0.8cm][c]{pull}),
\\[5pt]
p_{(x_1,x_2), (x_1-1,x_2)}& = \dfrac{\mu_1}{\mu_1+\lambda_1}, &
p_{(x_1,x_2), (x_1+1,x_2)}& = \dfrac{\lambda_1}{{\mu_1+\lambda_1}}, &
\,\,\,\mbox{if}  \quad {\cal P}\big((x_1,x_2)\big)& = (\makebox[0.8cm][c]{push},\makebox[0.8cm][c]{pull}),
\\[5pt]
p_{(x_1,x_2), (x_1,x_2-1)}& = \dfrac{\mu_2}{\mu_2+\lambda_2}, &
p_{(x_1,x_2), (x_1,x_2+1)}& = \dfrac{\lambda_2}{{\mu_2+\lambda_2}}, &
\,\,\,\mbox{if}  \quad {\cal P}\big((x_1,x_2)\big)& = (\makebox[0.8cm][c]{pull},\makebox[0.8cm][c]{push}).
\end{align*}
A given class of $\{X_n\}$ is positive recurrent, if and only if the associated class in the Markov jump process is positive recurrent. This holds since all transitions rates in the chain are bounded from above by $\lambda_1+\lambda_2+\mu_1+\mu_2$ and thus the Markov jump process is non-explosive (see for example \cite{norris1997mc}, Theorem~3.5.3).
 
In the case where $\lambda_1=\mu_1=\mu_2=\lambda_2$ the process $\{X_n,n\ge0\}$ is a simple symmetric random walk on a degenerate ({\em non}-random) environment similar to the random walks studied in \cite{holmes2011random}. 

\begin{theorem}
\label{thm:push-pull-Instability}
The critical push-pull network is non-stabilizable.
\end{theorem}
\begin{proof}
Assume that there exists a policy ${\cal P}$ such that ${X_n}$ has a positive recurrent class, ${\cal B} \subset {\cal S}$. 
Define,
\[ 
g\big( (x_1,x_2) \big) := \lambda_1 x_1 - \lambda_2 x_2.
\]
Under the same probability space, define $Z_n = g(X_n)$. It is readily verified by the transition probabilities above that $\Expect[Z_{n+1} | \sigma(X_0,\ldots,X_n)] = Z_n$ and $\Expect[|Z_n|] < \infty$ hence $\{Z_n\}$ is a martingale (for any choice of ${\cal P}$).

Pick now two arbitrary states $\mathbf{x},\mathbf{y}  \in {\cal B}$ such that $g(\mathbf{x}) \neq g(\mathbf{y})$. It is obvious 
that two such states exist under the assumption of positive recurrence of ${\cal B}$ and the form of $g(\cdot)$. Set $X_0 = \mathbf{x}$ w.p.~1 and  define the stopping time $T = \inf\{n\ge 0 : X_n = \mathbf{y} \}$. Since ${\cal B}$ is assumed to be positive-recurrent, $\Expect[T]<\infty$. As can be verified with the triangle inequality, $|Z_{n+1}-Z_{n}|<1$ a.s.\ and thus, by the optional stopping theorem (see for example \cite{williams1991pm}, Section 10.10), $\Expect[Z_T] = \Expect[Z_0]$. Hence,
\begin{equation}
\label{eq:Contra}
g(\mathbf{x}) = \Expect[Z_0] = \Expect[Z_T] = g(\mathbf{y});
\end{equation}
a contradiction. Hence there cannot exist a positive recurrent class ${\cal B}$.
\end{proof}

\section{Generalizations}

The key idea of the proof of Theorem~\ref{thm:push-pull-Instability} is to find a function $g(\cdot)$ over ${\cal S}$ that is a martingale (harmonic function) for any possible policy ${\cal P}$. The fact that a linear harmonic function was successfully employed in the critical push-pull network suggests that the method can be generalized to higher dimensional queueing networks. 

In this section we first present general sufficient criteria for non-stabilizability of objects that we refer to as {\em homogeneous controlled  queueing networks}.  Then, we employ this result to show that two special cases that are structured generalizations of the critical push-pull network are non-stabilizable.

Our object of study is a controlled discrete time Markov chain $\{X_n,\,~n\ge~0\}$ with state space ${\cal S} =\mathbb{Z}_+^M$. Denote by ${\cal A}(\mathbf{z})$ the set of actions that can be performed from any state $\mathbf{z} \in {\cal S}$,. 
Let ${\cal A} =  \bigcup_{\mathbf{z} \in {\cal S}} {\cal A}(\mathbf{z}) $
be the set of all possible actions and assume that $|{\cal A}| = L < \infty$. A~deterministic policy is then a function,
$
{\cal P} : {\cal S} \rightarrow {\cal A},
$
with the restriction ${\cal P}(\mathbf{z}) \in {\cal A}(\mathbf{z})$. It is assumed that under such policies $X_n \in {\cal S}$ for all $n$. Transition probabilities depend on the selected action $a \in {\cal A}(\mathbf{z})$ and are assumed to be specified 
by $\Prob_a \big( X_{n+1}=\cdot      \,|\, X_n=\mathbf{z} \big)$.

These objects are termed {\em homogenous} since we assume that,
\[
\tilde{P}_a(\mathbf{x}) := \Prob_a \big( X_{n+1}- \mathbf{z} = \mathbf{x}     \,|\, X_n=\mathbf{z} \big),
\]
 is independent of $\mathbf{z}$. Further they resemble {\em queueing networks} due to the transition structure that we describe now. Let $\{\mathbf{e}_i\}_{i=1}^M$ denote  the canonical row vectors in $\RL^M$ and define the set of possible transitions,
\[ 
{\cal D} := \{{\mathbf x} \in {\mathbb R}^M \, : \, \tilde{P}_a(\mathbf{x}) > 0\,\, \mbox{for some}\,\, a \in {\cal A}\}.
\]
We assume elements of ${{\cal D}}$ are of one of these three forms: $\mathbf{e}_i$, $-\mathbf{e}_i$ or $\mathbf{e}_i-\mathbf{e}_j$ for $i\neq j$.  It is further useful to define 
\[
{\cal I}_1 := \{i:  \mathbf{e}_i \in {{\cal D}} \,\, \mbox{or} \,\,  -\mathbf{e}_i \in {{\cal D}}   \},
\qquad
{\cal I}_2:= \{(i,j):  \mathbf{e}_i-\mathbf{e_j} \in {{\cal D}}   \}.
\]
Elements of ${\cal I}_1$  correspond to arrivals or departures of jobs into the network. Elements of ${\cal I}_2$ correspond to movements of jobs between queues in the network.

Our goal is to establish non-stabilizability in certain parameter cases. Similar to Definition~\ref{def:nonStab} we say that these more general networks are {\em non-stabilizable} if there does not exist a policy ${\cal P}$ that induces a positive recurrent class that is reached w.p.\ $1$. Mirroring Theorem~\ref{thm:push-pull-Instability}, our key step used to prove non-stabilizability is to find
a function $g:\cal S\rightarrow \RL$ that is harmonic for all policies. The function $g$ should satisfy
\begin{equation}\label{martingale}
\Exp_a [ g(X_{n+1}) - g(X_n)  \,|\, X_n = \mathbf{z}]  = 0,\quad  \mathbf{z} \in {\cal S},\,\, 
a \in {\cal A}(\mathbf{z}).
\end{equation}
The appeal of being homogenous and having a finite action space 
$\mathcal{A}=\{a_1,\dots,a_L\}$ is that the number of these equations reduces to $L$:
\[
\Exp_{a_i} [ g(X_{n+1}) - g(X_n) \,|\, X_n]  = 0,\quad i=1,\dots,L.
\]
It is now sensible to restrict the search of harmonic functions to the class of 
linear functions, $g(\mathbf{z}) = \boldsymbol{\alpha}' \mathbf{z}$ with
$\boldsymbol{\alpha} \in \mathbb{R}^M$ and $\boldsymbol{\alpha} \neq \mathbf{0}$. In this case, proving that $g$ is harmonic reduces to finding  
$\boldsymbol{\alpha} \neq \mathbf{0}$ such that,
\begin{equation}
\label{eq:DeltaAlphaSingle}
\boldsymbol{\alpha}' \, \boldsymbol{\Delta}_i = 0,\quad i=1,\dots,L,
\end{equation}
where $\boldsymbol{\Delta}_i := \Expect_{a_i} [ X_{n+1} - X_n \,|\, X_n ]$ is the \emph{drift vector of 
action} $a_i$. We define the $L \times M$ dimensional 
{\em action drift matrix} $\,\mathbf{D}$ as having rows $\boldsymbol\Delta_i',\ i=1,\dots, L$. Therefore (\ref{eq:DeltaAlphaSingle}) becomes
\begin{equation}
\label{eq:DalphaE0}
\mathbf{D} \boldsymbol{\alpha} = \mathbf{0}. 
\end{equation}
We are now in a position to state a sufficient condition for non-stabilizability:
\begin{theorem}
\label{thm:genHCRWinstab}
Consider a homogeneous controlled queueing network $\{X_n,\,~n\ge~0\}$ with action 
drift matrix $\mathbf{D}$. Then it is non-stabilizable if
\begin{equation}
\label{eq:rankCondition}
\mathrm{rank} \big ( \mathbf{D} \big) < M,
\end{equation}
and the following non-degeneracy condition holds: The system \eqref{eq:DalphaE0} has a solution $\boldsymbol\alpha$ such that for every  ${\mathbf x} \in{\cal S}$ and every $a \in {\cal A}(\mathbf{x})$,  $\Prob_a \big(  \boldsymbol{\alpha}' X_1 \neq \boldsymbol{\alpha}'\mathbf{x} \,|\, X_0 = \mathbf{x} ) > 0$.
\end{theorem}
\begin{proof}
The proof follows the exact same lines as the proof of Theorem~\ref{thm:push-pull-Instability}. The rank condition implies that (\ref{eq:DalphaE0}) has a non-trivial solution $\boldsymbol{\alpha}$, so there exists a non-trivial linear function $g(\cdot)$ that is harmonic for all policies.  The non-degeneracy condition implies that for every ${\mathbf x} \in{\cal S}$ there exists 
a state ${\mathbf y} \in{\cal S}$ which is reachable from $\mathbf{x}$ and is such that $g(\mathbf{x}) \neq g(\mathbf{y})$. Hence, if it is assumed that there exists a positive recurrent class ${\cal B}$, then
a contradiction as in (\ref{eq:Contra}) can be reached. The additional condition needed for the optional stopping theorem requiring that $|Z_{n+1}-Z_n|$ be bounded is guaranteed by the form of~${\cal D}$.
\end{proof}

Observe that the technical non-degeneracy condition actually implies \eqref{eq:rankCondition}, yet we view the rank condition \eqref{eq:rankCondition} as the central pillar for establishing non-stabilizability. The  lemma below puts forward sufficient conditions for satisfying the non-degeneracy condition:
\begin{lemma}
\label{lemma:non-deg-1}
Consider a homogeneous controlled queueing network for which there exists an $\boldsymbol{\alpha} \neq {\mathbf 0}$ solving \eqref{eq:DalphaE0}. If for every $i \in {\cal I}_1$, $\alpha_i \neq 0$ and for every $(i,j) \in {\cal I}_2$, $\alpha_i \neq \alpha_j$ then the non-degeneracy condition holds.
\end{lemma}
\begin{proof}
For every initial condition $X_0$ and every action in ${\cal A}(X_0)$, almost surely one of these two events takes place: (1) A single coordinate, $i^* \in {\cal I}_1$ changes. (2) Two coordinates, denoted $(i^*,j^*) \in {\cal I}_2$ change.
If (1) occurs then $|\boldsymbol\alpha' (X_1 - X_0)| = |\alpha_{i^*}| \neq 0$.
If (2) occurs then $\boldsymbol\alpha' (X_1 - X_0) = \alpha_{i^*} - \alpha_{j^*}  \neq 0$.
So in any case,   $\boldsymbol\alpha'X_1 \neq \boldsymbol\alpha' X_0$ a.s.
\end{proof}

\subsection{A Push-Pull Ring with an Even Number of Servers}
\begin{figure}[ht]
\centering
\includegraphics[width=2.8in]{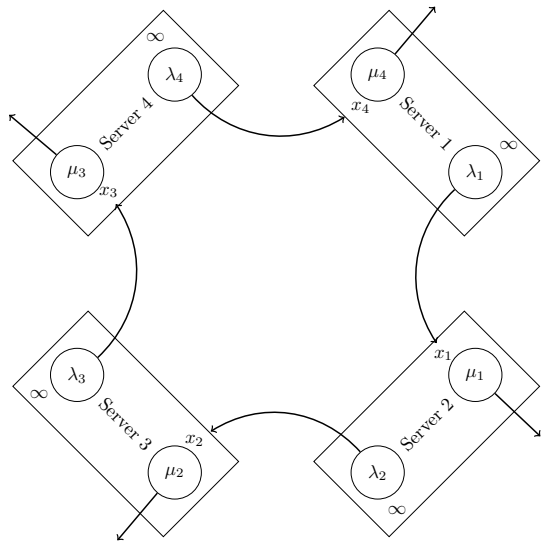}
\caption{A Ring of $M=4$ Push-Pull servers.}
\label{fig.ring}
\end{figure}
One way of generalizing the push-pull network is to allow more servers as in Figure~\ref{fig.ring}. In this network there are $M\ge2$ servers with the same number of streams. Each stream has a push operation at rate $\lambda_i$ and a pull operation at rate $\mu_i$. Our index notation implies that server $i$ has a choice between {\em push} to stream $i$ (at rate $\lambda_i$) or {\em pull} from stream $i-1$ (at rate $\mu_{i-1}$). All index arithmetics are modulo $M$ on $\{1,\ldots,M\}$.

Similar to the push-pull network (the special case with $M=2$), a non-idling policy is a mapping
$
{\cal P} : {\cal S} \rightarrow \{\mbox{push},\mbox{pull}\}^M,
$
with the restriction that the pull operation $i$ (on server $i+1$) can only be performed if $x_i>0$.

Stability properties of this model under a ``pull-priority'' policy (and general processing times) were investigated in \cite{GuoLefeberNazarathyWeissZhang11}, Section~5. In there it is shown that the network is stable if $\lambda_i<\mu_i,\, i=1,\ldots,M$ as well as if $M$ is odd and $\lambda_i>\mu_i, \, i=1,\ldots, M$, yet the ratio $\lambda_i/\mu_i$ is ``not too large'' (see \cite{GuoLefeberNazarathyWeissZhang11} for the full details). We further believe that for odd $M$, the network is stabilizable 
in the critical case ($\lambda_i=\mu_i$ for all $i$)
by means of a pull-priority policy. This has not been established in \cite{GuoLefeberNazarathyWeissZhang11}, yet stems from the same intuition appearing in \cite{GuoLefeberNazarathyWeissZhang11} based on the concept of a {\em mode}.

An interesting aspect of push-pull network rings is that \emph{even rings} (network rings with an even number of servers $M$) are drastically different to \emph{odd rings}. For example, the pull-priority policy cannot be stablilized for even rings in which $\lambda_i > \mu_i$ for all $i$: To see this, assume the state of the system is,
\[
\big( x_1, x_2,\ldots,x_n \big) = \big(+,0,+,0,\ldots,+,0\big),
\]
where `$+$' indicates a strictly positive quantity. Then, the selected action under pull-priority is
\[
\big(\mbox{push},\mbox{pull},
\mbox{push},\mbox{pull},
\ldots,
\mbox{push},\mbox{pull}
\big).
\]
In this case there is a strictly positive probability that the `+' queues will grow without bound. The drastic differences between even and odd rings go beyond the pull-priority policy of \cite{GuoLefeberNazarathyWeissZhang11}. We now show:
\begin{theorem}
The critical push-pull ring with $M$ even is non-stabilizable.
\end{theorem}

\begin{proof}
The push-pull ring, described as a homogeneous controlled random walk, has action drift matrix $\mathbf{D}$ of dimension $2^M \times M$, ${\cal I}_1 = \{1,\ldots,M\}$ and ${\cal I}_2 = \emptyset$.  We apply Theorem~\ref{thm:genHCRWinstab} and Lemma~\ref{lemma:non-deg-1}. The rank and non-degeneracy conditions are ensured by the $\boldsymbol\alpha$ found in Lemma~\ref{lemma:ringRankD} below.
\end{proof}

\begin{lemma}
In the critical case
\label{lemma:ringRankD}
\[
\mathrm{rank} \big( \mathbf{ D} \big) = 
\left\{
\begin{array}{cl}
M & \mbox{if}\quad       M \mbox{ is}\,\, \mbox{odd}, \\
M-1 & \mbox{if}\quad     M \mbox{ is}\,\, \mbox{even}. \\
\end{array}
\right.
\]
Further, when $M$ is even a solution of \eqref{eq:DalphaE0} is
  \begin{equation*}
  \bfalpha=\begin{pmatrix}
    \makebox[1.25cm][c]{$+\lambda_1^{-1}$}
   &\makebox[1.25cm][c]{$-\lambda_2^{-1}$}
   &\makebox[1.25cm][c]{$+\lambda_3^{-1}$}
   &\makebox[1.25cm][c]{$-\lambda_4^{-1}$}
   &\makebox[1.25cm][c]{$\dots$}
   &\makebox[1.5cm][c]{$+\lambda_{M-1}^{-1}$}
   &\makebox[1.25cm][c]{$-\lambda_{M}^{-1}$}
   \end{pmatrix}'.
 \end{equation*}
\end{lemma}
\begin{proof}

Define $\beta_j :=\lambda_j-\mu_j$, $j=1,\ldots,M$ and for a given action $i \in \{1,\ldots,2^M\}$ let $r_i$ be the sum of rates associated with the action. Then $r_i\,\boldsymbol\Delta_i$ has entries that depend on the action $i$ as follows:
\begin{center}
\begin{tabular}{|cc|cc|}
\hline
\multicolumn{2}{|c|}{Operation on:}&&\\
Server $j-1$        &   Server $j$      &&       Entry $j$ \\
\hline
push    &   push    &&   $\lambda_j$ \\
push    &   pull    &&   $\beta_j$ \\
pull    &   push    &&   $0$ \\
pull    &   pull    &&   $-\mu_j$ \\
\hline
\end{tabular}
\end{center}
This, in turn, implies the following relation between consecutive entries:
\begin{center}
\begin{tabular}{|r@{}clr@{}l|}
\hline
E&ntry $j$ && Entry&\ $j+1$\\
\hline
&$\lambda_j$   & $\Rightarrow$ 
 & \makebox[1.25cm][c]{$\lambda_{j+1}$}  o&r \makebox[1.25cm][c]{$\beta_{j+1}$} \\
&$-\mu_j$    & $\Rightarrow$   
 & \makebox[1.25cm][c]{$-\mu_{j+1}$}     o&r \makebox[1.25cm][c]{$0$} \\
&$0$           & $\Rightarrow$ 
 & \makebox[1.25cm][c]{$\beta_{j+1}$}    o&r \makebox[1.25cm][c]{$\lambda_{j+1}$} \\
&$\beta_j$   &  $\Rightarrow$  
 & \makebox[1.25cm][c]{$0$}              o&r \makebox[1.25cm][c]{$-\mu_{j+1}$} \\
\hline
\end{tabular}
\end{center}
Define the matrix,
$\hat{\mathbf{D}} = \mbox{sign} \big( \mathbf{D} \big)$,
where the function $\mbox{sign}(\cdot)$ is taken element-wise. In the critical case ($\beta_i = 0$, $i=1,\ldots,M$) it is evident that $\mathrm{rank}(\hat{\mathbf{D}}) = \mathrm{rank}(\mathbf{D})$ since $\mathbf{D}_{i,j} = \hat{\mathbf{D}}_{i,j} \lambda_j/r_j $. Further, by considering the structure of consecutive entries in the table above, it is evident that in each row of $\hat{\mathbf{D}}$:
\begin{enumerate}[(i)]
 \item\label{p4} The number of $0$'s separating two non-zero entries with the {\bf opposite sign} is \textbf{odd}.
 \item\label{p3} The number of $0$'s separating two non-zero entries with the {\bf same sign} is  \textbf{even}.
 \end{enumerate}
A consequence of (i) and (ii) is that the number of zero entries in each row of $\hat{\mathbf{D}}$ is even since there is an even number of 0-sequences that have an odd number of zeros. Hence in an odd/even ring there is an odd/even number of non-zero entries.

{\em $M$ odd case:}\\
Observe each of the vectors in $\{\mathbf{e}_i: i=1,\ldots,M\}$
is also a row of $\hat{\mathbf{D}}$ and therefore it has full rank. 

{\em $M$ even case:}\\
First, we prove that
 $\mathrm{rank}(\hat{\mathbf{D}})\ge M-1$.  Second, we show $\hat{{\mathbf D}} \hat{\boldsymbol{\alpha}} = {\mathbf 0}$, where,
  \begin{equation*}
  \hat{\bfalpha}=\begin{pmatrix}
    \makebox[1.25cm][c]{+1}
   &\makebox[1.25cm][c]{-1}
   &\makebox[1.25cm][c]{+1}
   &\makebox[1.25cm][c]{-1}
   &\makebox[1.25cm][c]{$\dots$}
   &\makebox[1.5cm][c]{+1}
   &\makebox[1.25cm][c]{-1}
   \end{pmatrix}',
 \end{equation*}
 this immediately implies that ${\mathbf D} \boldsymbol\alpha = {\mathbf 0}$.
 
 Consider the vectors $\mathbf{f}_i=\mathbf{e}_i+\mathbf{e}_{i+1}$,
 $i=1,\dots,M-1$, observing these are rows of $\hat{\mathbf{D}}$. Next, define $\mathbf{B}$ to be the matrix of size $(M-1)\times M$ with $i$'th row $\mathbf{f}_i$. Let $\mathbf{B}_{-1}$ be a square matrix of size $(M-1)\times(M-1)$
 obtained by deleting the first column of $\mathbf{B}$.  Clearly,  $\mathbf{B}_{-1}$ has full rank
 (the determinant of a lower triangular matrix is the product of its diagonal elements). Hence 
 $\mathrm{rank}(\hat{\mathbf{D}})\ge M-1$. 

To show $\mathrm{rank}(\hat{\mathbf{D}})<M$ we show that $\hat{\boldsymbol\Delta}^{'}\hat{\bfalpha}=0$ for any row $\hat{\boldsymbol\Delta} = (\delta_1,\ldots,\delta_M)'$ of the matrix $\hat{\mathbf{D}}$.

 Let $i_k$ be the index of the $k$-th non-zero entry of $\hat{\boldsymbol\Delta}$ starting from index $1$, and $m$
 its total number of non-zero entries.  Since $m$ is even, then we can write
 \begin{equation*}
  \hat{\boldsymbol\Delta}^{'}\hat{\bfalpha}=\sum_{k=1}^m \delta_{i_k}\hat{\alpha}_{i_k}=
  \sum_{k=1}^{m/2}\big(\delta_{i_{2k-1}}\hat{\alpha}_{i_{2k-1}}+\delta_{i_{2k}}\hat{\alpha}_{i_{2k}}\big)
 \end{equation*}
 Define $\ell_k:=i_{2k}-i_{2k-1}-1$ the number of zeros between the 
 non zero consecutive entries $\delta_{i_{2k-1}}$ and $\delta_{i_{2k}}$. By the definition
 of $\widehat{\boldsymbol\Delta}$ and $\hat{\bfalpha}$ there are only two possibilities
 \begin{align*}
  \ell_k=\left\{
 \begin{array}{ccrlcrl}
  \text{odd}&\Longrightarrow&\delta_{i_{2k-1}}=
   &\makebox[0.75cm][c]{$-\delta_{i_{2k}}$}&\text{and}&\hat{\alpha}_{i_{2k-1}}=
   &\makebox[0.75cm][c]{$\hat{\alpha}_{i_{2k}}$,}\\[.5cm]
  \text{even}&\Longrightarrow&\delta_{i_{2k-1}}=
   &\makebox[0.75cm][c]{$\delta_{i_{2k}}$}&\text{and}&\hat{\alpha}_{i_{2k-1}}=
   &\makebox[0.75cm][c]{$-\hat{\alpha}_{i_{2k}}$.}
 \end{array}\right.
 \end{align*}
 This implies that
 \begin{equation*}
  \delta_{i_{2k-1}}\hat{\alpha}_{i_{2k-1}}+\delta_{i_{2k}}\hat{\alpha}_{i_{2k}}=0,\qquad k=1,\dots,m/2,
 \end{equation*}
 proving that $\hat{\bfalpha}$ is a solution of $\hat{\mathbf{D}}\hat{\bfalpha}=\mathbf{0}$.
 \end{proof}

\begin{figure}[ht]
\centering
\includegraphics[width=2.8in]{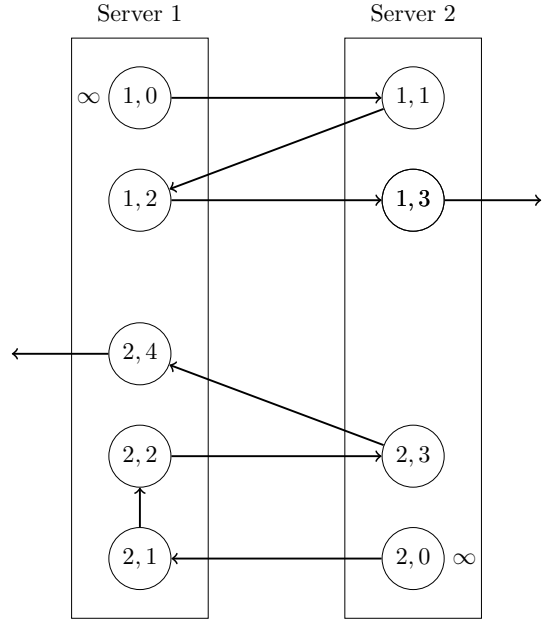}
\caption{A network of re-entrant lines on two servers.  In this example $S=2$, $n_1=3$ and $n_2 =4$. Hence the number of buffers is $M=7$. Further the number of actions that a control policy may choose is $L=20$.}
\label{fig.rlinePushPull}
\end{figure}
\subsection{Re-entrant Lines on Two Servers}

We now consider a network as that appearing in Figure~\ref{fig.rlinePushPull} essentially generalizing the push-pull by allowing streams of more than two steps that may re-enter servers multiple times. Denote the number of streams by $S \ge 1$.  Each stream $i \in \{1,\ldots,S\}$, has $n_i+1$ operations labeled by $(i,0),\ldots,(i,n_i)$, where the operation $(i,0)$ is associated with an infinite supply of materials and the other operations are associated with queues. Assume $n_i\ge 1$.  Hence the number of queues in the system is,
\[
M= \sum_{i=1}^S n_i,
\]
and the total number of operations is $M+S$. As in the push-pull network there are two servers, labeled $1$ and $2$. Each operation $(i,j)$ is performed by a unique server denoted $\sigma(i,j) \in \{1,2\}$. A control policy is then a rule based on the number of jobs in each of the queues, indicating for each server what operation it needs to work on.

The processing rate of operation $(i,j)$ is denoted by $\mu_{i,j}>0$. 
It is useful to denote, for each stream, $i \in \{1,\ldots,S\}$, the set of operations on server $\ell \in \{1,2\}$ by,
\[
C_\ell(i) := \big\{ j\in \{0,\ldots,n_i\} \,\, :\, \sigma(i,j) = \ell \big\}.
\] Motivated by \cite{GuoLefeberNazarathyWeissZhang11}, Section~4, we say that the network is {\em critical} if:
\begin{equation}
\label{eq:rlineCrit}
\sum_{j \in C_1(i)} \mu_{i,j}^{-1} = \sum_{j \in C_2(i)} \mu_{i,j}^{-1}, \quad i=1,\ldots,S.
\end{equation}
This means that for any job, the mean processing duration of service required by the job on each of the servers is equal.

The  case of $S=1$ is typically called the re-entrant line (with infinite supply). It was analyzed in \cite{GuoLefeberNazarathyWeissZhang11}, Section~3 under a last buffer first serve policy. The case of $S=2$  is a direct generalization of the push-pull network. It was analyzed in \cite{GuoLefeberNazarathyWeissZhang11}, Section~4 under a specific priority policy.

For notational purpuses, it is useful to order the $M$ queues using some arbitrary permutation
of the set $\{1,\ldots,M\}$. Denote $\bar{k}(i,j)$ to be the queue served by operation $(i,j)$, $i \in \{1,\ldots,S\}$, $j \in \{1,\ldots,n_i\}$ (note that operations $(i,0)$ do not have an associated queue). Further denote $\bar{i}(k)$ and $\bar{j}(k)$ as the inverse mappings, i.e., 
$\overline{k}\big(\, \overline{i}(k), \overline{j}(k) \big) = k.$

Similar to the push-pull ring of the previous sub-section, it is straightforward to model this network as a controlled random walk.
In doing so, it is useful to partition the set ${\cal I}_1$ into ${\cal I}_{1,+}$, the queues fed by {\em push} operations at start of each stream and ${\cal I}_{1,-}$, the queues drained by {\em pull} operations at the end of each stream.
Further, the set ${\cal I}_2$ is associated with operations that are neither at the start or the end of the stream, involving movement of jobs between buffers without changing the total number of jobs in the system.
We now have,
\begin{eqnarray*}
{\cal I}_{1,+} &=& \big\{ k \in \{1,\ldots,M\} \,:\, \bar{j}(k) = 1 \big \},\\
{\cal I}_{1,-} &=& \big\{ k \in \{1,\ldots,M\} \,:\, \bar{j}(k) = n_{\bar{i}(k)} \big \}, \\
{\cal I}_2 \,\,\,\, &=& \big\{ (k,k')   \in \{1,\ldots,M\}^2 \,:\,  \bar{i}(k) = \bar{i}(k'),\,   \bar{j}(k) = \bar{j}(k')-1         \big\}.
\end{eqnarray*}

Denote $L_\ell =  \sum_{i=1}^S | C_\ell(i)|$ for $\ell=1,2$. Then the total number of actions is $L= L_1\,L_2$ and the action drift matrix $\mathbf{D}$ is of dimension $L \times M$ as is consistent with previous notation. Denote for each operation $(i,j)$ and each $k \in \{1,\ldots,M\}$:
\[
 {\hat{\Delta}}_k(i,j)  =
 \begin{cases}
  \begin{cases}
   \mu_{i,0} &  k = \bar{k}(i,1), \\
   0       & \mbox{otherwise},
  \end{cases}
  &
  \,\, j=0, \\[14pt]
   \begin{cases}
   -\mu_{i,j} & k = \bar{k}(i,j), \\
   \mu_{i,j} & k = \bar{k}(i,j+1), \\
   0       & \mbox{otherwise},
  \end{cases}
  &
  \,\,  j \in \{1,\ldots,n_i-1\}, \\[24pt]
   \begin{cases}
   -\mu_{i,n_i} & k  = \bar{k}(i,n_i), \\
   0       & \mbox{otherwise},
  \end{cases}
  &
  \,\,  j=n_i. \\
 \end{cases}
\]


Let $\boldsymbol{\hat{\Delta}}(i,j) \in {\mathbb R}^M$ be the vector of these elements.
Now each row of ${\mathbf D}$ corresponds to two actions, $(i_1,j_1)$ such that $\sigma(i_1,j_1) = 1$ and $(i_2,j_2)$ such that $\sigma(i_2,j_2)=2$.  We denote this row by $\boldsymbol{\Delta}(i_1,j_1,i_2,j_2)'$ with,
\[
\boldsymbol{\Delta}(i_1,j_1,i_2,j_2) := \frac{1}{||\boldsymbol{\hat\Delta}(i_1,j_1) + \boldsymbol{\hat\Delta}(i_2,j_2)||_1}(\boldsymbol{\hat\Delta}(i_1,j_1) + \boldsymbol{\hat\Delta}(i_2,j_2)).
\]
Having defined the controlled random walk we are now ready to prove that it is non-stabalizable.
\begin{theorem}
The critical network is non-stabilizable.
\end{theorem}

\begin{proof}
We find an $\boldsymbol\alpha \in {\mathbb R}^M$ such that $\mathbf{D} \boldsymbol \alpha = \mathbf{0}$. The elements of $\boldsymbol\alpha$ are,
\begin{equation*}
\alpha_k = \sum_{j=0}^{\bar{j}(k)-1}\frac{1}{\mu_{\bar{i}(k),j}}(-1)^{\sigma(\bar{i}(k),j )}.
\end{equation*}
For any $i \in \{1,\ldots,S\}$, it is straightforward to verify that if $j=0$,
\[
\boldsymbol{\hat{\Delta}}(i,j)' \boldsymbol\alpha = \mu_{i,0}\frac{1}{\mu_{i,0}}(-1)^{\sigma(i,j )} =  (-1)^{\sigma(i,j )}.
\]
Further, if $j \in \{1,\ldots,n_{i}-1\}$ then, 
\[
\boldsymbol{\hat{\Delta}}(i,j)'\boldsymbol\alpha  = -\mu_{i,j} 
 \sum_{j'=0}^{j-1}\frac{1}{\mu_{i,j'}}(-1)^{\sigma(i,j' )}
 + \mu_{i,j}
 \sum_{j'=0}^{j}\frac{1}{\mu_{i,j'}}(-1)^{\sigma(i,j' )}
 =
 (-1)^ {\sigma(i,j)}.
 \]
 Further if $j = n_{i}$,
 \begin{eqnarray*}
\boldsymbol{\hat\Delta}(i,j)' \boldsymbol\alpha
&=&
 -\mu_{i,n_{i}}  \sum_{j'=0}^{n_{i}-1}\frac{1}{\mu_{i,j'}} (-1) ^{\sigma(i,j' )}\\
&=&
-\mu_{i,n_{i}} 
\bigg(
\sum_{j' \in C_2(i)} \mu_{i,j'}^{-1}
-
\sum_{j' \in C_1(i)} \mu_{i,j'}^{-1}
-
\mu_{i,n_i}^{-1}
(-1)^{\sigma(i,n_i)}
\bigg) =
(-1)^{\sigma(i,n_i)}.
 \end{eqnarray*}
 In the last equality we use the fact that the network is critical, \eqref{eq:rlineCrit}.  Hence for any $(i,j)$,
 \[
 \boldsymbol{\hat{\Delta}}(i,j)' \boldsymbol\alpha = (-1)^{\sigma(i,j)} .
 \]
Now for any row of $\mathbf{D}$, i.e. for any $(i_1,j_1)$, $(i_2,j_2)$ such that $\sigma(i_1,j_1)=1$, $\sigma(i_2,j_2)=2$ we have that,
{\small
\[
\boldsymbol{\alpha}' \boldsymbol{\Delta}(i_1,j_1,i_2,j_2) =
\frac{1}{||\boldsymbol{\hat\Delta}(i_1,j_1) + \boldsymbol{\hat\Delta}(i_2,j_2)||_1}
\big(
(-1)^{\sigma(i_1,j_1)} + (-1)^{\sigma(i_2,j_2)}
\big)
 = 0,
\]
}
as needed. Since $\boldsymbol\alpha \neq \mathbf{0}$, the rank condition of Theorem~\ref{thm:genHCRWinstab} holds.

To see that the non-degeneracy condition is satisfied, we use Lemma~\ref{lemma:non-deg-1}. 
First, for every $k \in {\cal I}_{1,+}$, it is evident that $\alpha_k \neq 0$ .
Further, for  every $k \in {\cal I}_{1,-}$, denote $i=\bar{i}(k)$, then
\[
\alpha_k = \sum_{j=0}^{n_i-1} \mu_{i,j}^{-1} (-1)^{\sigma(i,j)}=
\sum_{j \in C_2(i)} \mu_{i,j}^{-1}
-
\sum_{j \in C_1(i)} \mu_{i,j}^{-1}
-
\mu_{i,n_i}^{-1}
(-1)^{\sigma(i,n_i)}
\neq 0,
\]
where the last inequality follows from the fact the network is critical.
Finally, for every pair $(k,k') \in {\cal I}_2$ we verify
\[
\alpha_{k'} - \alpha_{k} = \frac{1}{\mu_{\bar{i}(k),\bar{j}(k)}} (-1)^{\sigma\left(\bar{i}(k),\bar{j}(k)\right)} \neq 0,
\] 
as required.
\end{proof}

\section{Closing Comments}

The interest in non-stabilizability of the fully-utilizing critical push-pull network stems from the fact that in symmetric non-critical cases it can be stabilized. This is true for both $\lambda<\mu$ and $\lambda>\mu$ in contrast to traditional queueing networks.  Hence our contribution that the network is non-stabilizable for $\lambda=\mu$ is of interest.

The consequences to more general networks as appearing in Section~3 are interesting in their own right, since Theorem~\ref{thm:genHCRWinstab} provides a general sufficient condition for non-stabilizability of quite general models.

It should be noted that the class of policies, $\{ {\cal P}\}$, we consider is that of stationary deterministic policies. One can also consider randomized policies where an action $a \in {\cal A}$ is taken with a given probability. Extending the non-stabilizability results to this case is straightforward and carries no surprise. Further, using continuous time martingales, it is merely technical matter to extend to non-stationary policies.

As opposed to non-deterministic and/or non-stationary policies, a non-trivial extension is to consider networks with general processing times (not necessarily exponentially distributed). In this case one way of providing a Markovian description to a network is based on residual service times (see for example \cite{GuoLefeberNazarathyWeissZhang11} and references there-in).  In that case, we conjecture that the non-stabilizability results of this paper carry over, yet proving such results may require a different set of tools than the one we have used here.

\bibliographystyle{elsarticle-num}
\bibliography{NazNanSalbib}
\end{document}